\documentclass[11pt, a4paper, oneside, reqno]{amsart}

\usepackage  [latin1 ]{inputenc}
\usepackage{amsmath, amsthm, amssymb, amsfonts}
\usepackage[foot]{amsaddr}
\usepackage{algorithm2e}
\usepackage[english]{babel}
\usepackage{bbding}
\usepackage{booktabs}
\usepackage{cite}
\usepackage[colorlinks=true, linkcolor = blue, citecolor = dgreen]{hyperref}
\usepackage{crossreftools}
\usepackage{cleveref}
\usepackage{color}
\usepackage{diagbox}
\usepackage{dsfont}
\usepackage[shortlabels]{enumitem}
\usepackage{eucal}
\usepackage{float}
\usepackage{framed}
\usepackage[a4paper]{geometry} 
\usepackage{graphicx}
\usepackage{mathrsfs}
\usepackage[all]{xy} 
\usepackage[numbers]{natbib}
\usepackage{rotating}
\usepackage{setspace}
\usepackage{stmaryrd}
\usepackage{subcaption}
\usepackage{xspace}
\usepackage[dvipsnames]{xcolor}

\setlist{leftmargin=*, topsep=0.5em, parsep=0pt, itemsep=1em, labelindent=0pt, align=left}

\theoremstyle{definition}
\newtheorem{theorem}{Theorem}[section]

\newtheorem{definition}{Definition}[section]

\newtheorem{proposition}{Proposition}[section]
\newtheorem{remark}{Remark}[section]

\newtheorem{lemma}{Lemma}[section]

\numberwithin{equation}{section}

\newcommand{\Pro}{\mathbb{P}}

\newcommand{\Ex}{\mathbb{E}}
\newcommand{\Var}{\mathrm{Var}}
\newcommand{\Cov}{\mathrm{Cov}}

\newcommand{\sgn}{\mathrm{sgn}}
\newcommand{\EE}{\mathcal{E}}

\newcommand{\GG}{\mathcal{G}}
\newcommand{\HH}{\mathcal{H}}
\newcommand{\NN}{\mathcal{N}}

\newcommand{\boH}{\textbf{H}}
\newcommand{\bH}{\bar H}
\newcommand{\tH}{\tilde H}
\newcommand{\hk}{\hat k}
\newcommand{\dd}{\mathrm {d}}

\newcommand{\reals}{\mathbb{R}}

\newcommand{\naturals}{\mathbb{N}}

\newcommand{\vep}{\varepsilon}

\newcommand{\gind}[2]{\textbf{1}_{#1}(#2)}
\newcommand{\rind}[1]{\textbf{1}_{#1}}

\numberwithin{table}{section}
\definecolor{red}{HTML}{D62728}
\definecolor{blue}{RGB}{ 0, 109, 219}
\definecolor{dgreen}{rgb}{0,.8,0}

\begin{document}
\title[]{Existence, renormalization, and regularity properties of higher order derivatives of self-intersection local time of fractional Brownian motion}
\author{Kaustav Das$^\dagger$}
\address{$^\dagger$School of Mathematics, Monash University, Victoria, 3800 Australia.}
\author{Greg Markowsky$^\dagger$}
\email{kaustav.das@monash.edu, greg.markowsky@monash.edu}
\date{}

 \DeclareGraphicsExtensions{.pdf,jpg}
\vspace{12pt}

\crefformat{enumi}{item~(#2#1#3)}
\Crefformat{enumi}{Item~(#2#1#3)}
\crefrangeformat{enumi}{items~(#3#1#4) to~(#5#2#6)}
\Crefrangeformat{enumi}{Items~(#3#1#4) to~(#5#2#6)}

%
\begin{abstract}
In a recent paper by Yu (arXiv:2008.05633, 2020), higher order derivatives of self-intersection local time of fractional Brownian motion were defined, and existence over certain regions of the Hurst parameter $H$ was proved. Utilizing the Wiener chaos expansion, we provide new proofs of Yu's results, and show how a Varadhan-type renormalization can be used to extend the range of convergence for the even derivatives. \\[.6cm]

\hspace{-.42cm}Keywords: Self-intersection local time; derivatives of self-intersection local time; fractional Brownian motion; Wiener chaos.
\end{abstract}
\maketitle
\section{Introduction}
\label{sec:introduction}

Let $(B^H_t)$ be a fractional Brownian motion (fBm) with Hurst parameter $H \in (0,1)$. Namely, $(B^H_t)$ is a centred Gaussian process, continuous a.s., with covariance function 
\begin{align*}
	R_H(t,s) = \frac{1}{2}(s^{2H} + t^{2H} - |s - t|^{2H}), \quad s, t \geq 0. 
\end{align*}
The self-intersection local time of fBm is defined as
\begin{align*}
	\alpha_t(y) := \int_0^t \int_0^s \delta(B^H_s  - B^H_r - y) \dd r \dd s,
\end{align*}
where $\delta$ is the Dirac delta function. Intuitively, $\alpha_t(0)$ measures the amount of time that the process $(B^H_t)$ spends revisiting prior attained values on the interval $[0,t]$. \\

The object of focus in this paper is the $k$-th derivative of self-intersection local time (DSLT) of fBm, which was introduced in the interesting recent paper \citep{yu2020higher}. To be precise, we define
\begin{align*}
	\alpha_{t,\vep}(y):=\int_0^t \int_0^s f_\vep(B^H_s  - B^H_r - y) \dd r \dd s,
\end{align*}
where $f_\vep(x) := \frac{1}{\sqrt{2 \pi \vep}} e^{-x^2/(2\vep)}$ is the centred Gaussian density with variance $\vep$. Since $f_\vep \to \delta$ as $\vep \downarrow 0$ weakly, then one can think of $\alpha_{t,\vep}$ as an approximation of $\alpha_t$ for small $\vep$. Consider 
\begin{align*}
		\alpha^{(k)}_{t,\vep}(y):= (-1)^k \int_0^t \int_0^s f^{(k)}_\vep(B^H_s  - B^H_r - y) \dd r \dd s,
\end{align*}
where $f_\vep^{(k)}(x) := \frac{\dd^k }{\dd x^k} f_\vep(x)$. The $k$-th derivative of $\alpha_t$ is then naturally defined as a limit (in a sense to be specified later) of $\alpha^{(k)}_{t,\vep}$ as $\vep \downarrow 0$. Yu proved a number of facts related to $\alpha^{(k)}_{t}$ in \citep{yu2020higher}, and in fact defined the process in arbitrarily high dimensions with mixed partial derivatives allowed; however we will be content to remain in one dimension and consider $y=0$. For this reason, from here on in, we will write $\alpha_{t,\vep}^{(k)} \equiv \alpha_{t,\vep}^{(k)} (0)$ and $\alpha_t^{(k)}\equiv \alpha_t^{(k)} (0)$. We will address the following result:
\begin{theorem}[\citep{yu2020higher}] \label{yubaby} 
$\alpha^{(k)}_{t}$ exists in $L^2(\Omega)$ whenever

\begin{equation*}
H < \left \{ \begin{array}{ll}
\frac{2}{2k+1}, & \qquad  \mbox{if $k$ is odd,} \\
\frac{1}{k+1}, & \qquad \mbox{if $k$ is even.} \;
\end{array} \right.
\end{equation*}
\end{theorem}

We will apply the Wiener chaos expansion, which was not used for this purpose by Yu, in order to give a new proof of this theorem, and in fact will be able to prove the following extension.

\begin{theorem} \label{vara}
If $k\in \naturals$ is even and $\frac{1}{k+1} \leq H < \frac{2}{2k+1}$, then $\alpha^{(k)}_{t,\vep}- \Ex[\alpha^{(k)}_{t,\vep}]$ converges in $L^2(\Omega)$ as $\vep \downarrow 0$.
\end{theorem}

This kind of result is commonly referred to as a {\it Varadhan-type renormalization}, due to its origin in the important result from \citep{varadhan1969appendix}, which addressed the self-intersection local time of Brownian motion in two dimensions. We note also the similarity of \Cref{vara} with the main results in \citep{hu2005renormalized}, which addressed the self-intersection local time of fractional Brownian motion in arbitrarily many dimensions, and similarly showed convergence of the process, renormalized by subtracting the mean, for certain ranges of $H$. We remark that such a renormalization to increase the range of convergence in $H$ for the odd derivatives is not possible: since $\delta^{(k)}$ is odd in this case, the expectation will be 0. In the process of proving these results, we will also deduce the Wiener chaos decompositions of $\alpha^{(k)}_{t}$ and its renormalization (see the statements of \Crefrange{thm:convergenceL2odd}{thm:convergenceL2evenrenorm} below). \\

Our results fit naturally with a number of other results in the field, and for that reason in \Cref{sec:literaturereview} we provide a brief survey of the existing knowledge on this topic. In \Cref{sec:mainresults} we give precise statements of our main results, and in \Cref{sec:proofs} we prove them. The final sections, \Cref{appen:malliavin,appen:finitenessintegrals}, contain a discussion on the Malliavin calculus content that we employ within our methodology, as well as the proofs of a number of technical lemmas respectively.
\begin{remark}[Notation]
We will make use of the following notation extensively. 
\begin{framed}
\begin{enumerate}[label = (\roman*)]
\item $\boH := L^2([0,T])$ is the Hilbert space of square integrable functions on $[0,T]$ with respect to the Lebesgue measure.
\item $L^2(\Omega, \GG, \Pro) \equiv L^2(\Omega)$ denotes the Hilbert space of square integrable random variables on $(\Omega, \GG, \Pro)$, where $\GG$ is the $\sigma$-field generated by the isonormal Gaussian process $W^B = \{ W^B(\phi):= \int_0^T \phi_t \dd B_t : \phi \in \boH\}$.
\item $D_t := \{(r, s) \in \reals^2 : 0 \leq r \leq s \leq t \}$ is a simplex in $\reals^2$.
\item $\naturals = \{1, 2, \dots \}$ and $\naturals_0 := \naturals \cup \{0\}$.
\item $P = \{(n',k') \in \naturals_0^2: n'+k' = 0 \mod 2\}$ is the set of tuples of non-negative integers whose sum is even. 
\item We will utilize the following standard notation from the area of DSLT.
\begin{align*}
	\lambda &:= \Var( B_s^H - B_r^H) = |s - r|^{2H},\\
	\rho &:= \Var( B_{s'}^H - B_{r'}^H) = |s'-r'|^{2H}, \\
	\mu &:= \Cov(B_s^H - B_r^H, B_{s'}^H - B_{r'}^H) = \frac{1}{2}\left (|s - r'|^{2H} + |r - s'|^{2H} - |s - s'|^{2H} - |r - r'|^{2H}\right ).
\end{align*}
\item For a function $f : [a,b]^n \to \reals$, $I_n$ denotes the multiple Wiener integral, 
\begin{align*}
	I_n(f) = \int_{[a,b]^n} f(v_1, \dots, v_n) \dd B_{v_1} \dots \dd B_{v_n}.
\end{align*}
\end{enumerate}
\end{framed}
\end{remark}
\section{Literature review}
\label{sec:literaturereview}
The topic of self-intersection local times of stochastic processes has received a great deal of attention in recent decades. Originally studied for Brownian motion due to connections with theoretical physics (see \citep{edwards1965statistical, varadhan1969appendix, le1986proprietes, le1988fluctuation}), it has been generalized to a wide variety of processes, and has become a major focus of research for a number of theoretical probabilists. This paper lies in the intersection of several different research threads, and we will take the time to discuss these before stating our results.\\

The process which is our focus is fractional Brownian motion. The self-intersection local time of this process was first studied by Rosen in \citep{rosen1987intersection}, and this has led to a large literature on the subject, including \citep{hu2001self, hu2005renormalized, hu2008integral, grothaus2011self, nualart2007intersection, wu2010regularity, oliveira2011intersection, jiang2007collision, chen2011large, chen2011remarks, chen2018renormalized, jaramillo2019functional, rezgui2007renormalization, yu2019smoothness, evans2017spans}. A key difficulty with this process, in comparison to ordinary Brownian motion, is that in general the increments of the process are not independent, and this brings considerable complications into the calculations. The key method for overcoming this difficulty is to use the property of {\it local nondeterminism} (see \citep{berman1973local} and \Cref{lemma:boundingregions} below).\\

An important tool in the analysis of local times and intersection local times is the Wiener chaos expansion. This was first discovered in \citep{nualart1992chaos}, and the chaos expansion has since become a standard method in the field (see \citep{imkeller1995chaos, albeverio1997remark, hu2001self, hu2002chaos, jung2014tanaka, markowsky2012derivative, de2000renormalization, mataramvura2004donsker, bornales2016chaos}, to name a few examples).\\

The derivative of self-intersection local time of Brownian motion was also introduced by Rosen, in \citep{rosen2005derivatives}, and has blossomed into a research topic in its own right (\citep{hong2020derivatives, guo2018higher, guo2019higher, yan2014derivative, yan2015derivative, yan2017derivative, shi2017fractional, shi2020fractional, jaramillo2017asymptotic, jaramillo2019approximation, jung2014tanaka, jung2015holder, markowsky2008proof, markowsky2008renormalization, markowsky2012derivative, yu2019smoothness, yu2020asymptotic, yu2020higher}). The original motivation discussed by Rosen was the Tanaka-style formula

$$ 
\frac{1}{2}\alpha_{t}'(x) + \frac{1}{2}\sgn(x)t =
 \int_0^t L_s^{B_s - x}\dd B_s - \int_{0}^{t} \sgn
(B_{t}-B_{u} - x) \dd u. 
$$

Rosen only stated this formally, but later it was proved rigorously in \citep{markowsky2008proof}. The formal argument for this identity comes from applying Ito's Lemma to the non-differentiable function $\sgn$ and then integrating: since the derivatives of $\sgn$ are the Dirac delta function and its derivatives, we end up with DSLT in the equation. Similar formal calculations applied to $\delta$ itself will require higher order derivatives of self-intersection local time. For instance, the same argument applied directly to $\delta$ yields

$$
\frac{1}{2}\alpha_{t}''(x) + \frac{1}{2}\delta(x)t =
\int_0^t \int_0^s \delta'(B_s - B_u - x)\dd u\dd B_s - \int_{0}^{t} \delta(B_{t}-B_{u} - x) \dd u. 
$$

We note that the final term is simply the local time of the Brownian motion $\hat B_u := B_t-B_{t-u}$ for $0 \leq u \leq t$ at time $t$. However, each of the other terms is problematic: it is not clear that the stochastic integral exists, the term containing $\delta(x)$ without an integral is worrisome, and Yu's results show that $\alpha_{t}''(x)$ does not exist for Brownian motion. It is therefore unclear whether any meaning can be assigned to this identity; however, it is possible that some method of renormalization can be devised which clarifies it.

\section{Main results}
\label{sec:mainresults}
First, we define the following object, which will be utilized in the subsequent results: 
\begin{align*}
	g_{t,\vep}(n, k)& \equiv  g_{t,\vep}(n,k;v_1, \dots, v_n)  \\
	&=  \left (\frac{(-1)^{(3k+n)/2} (n+k - 1)!!  }{n! \sqrt{2 \pi}  }\right ) \left ( \int_0^t \int_0^s \frac{\prod_{j=1}^n K_{r,s}(v_j) }{[(s - r)^{2H} + \vep]^{(k+n+1)/2}} \dd r \dd s \right )\gind{P}{n,k},
\end{align*}
where $K_{r,s}$ is given in \cref{defnKthing}. The results pertaining to $\alpha_{t, \vep}^{(k)}$ and $\alpha_{t}^{(k)}$ will depend on the parity of $k$. 
\begin{lemma}
\label{lemma:wienerchaos}
Let $\hk, k' \in \naturals$. Then $\alpha_{t,\vep}^{(2\hk-1)}$ and $\alpha_{t,\vep}^{(2k')}$ admit the Wiener chaos expansions
\begin{align*}
	\alpha_{t,\vep}^{(2\hk-1)}&= \sum_{m=1}^\infty I_{2m - 1}(g_{t,\vep} (2m - 1, 2\hk-1)),
\end{align*}
and
\begin{align*}
	\alpha_{t,\vep}^{(2k')}&= \sum_{l=0}^\infty I_{2l}(g_{t,\vep} (2l, 2k')).
\end{align*}
\end{lemma}
\begin{theorem}
\label{thm:convergenceL2odd}
Let $k = 2\hk - 1$, where $\hk \in \naturals$. Suppose $H \in (0, \frac{2}{4\hk - 1})$. Then $\alpha_{t}^{(2\hk-1)}$ exists in $L^2(\Omega)$. Moreover, $\alpha_{t}^{(2\hk-1)}$ exhibits the Wiener chaos expansion
\begin{align*}
	\alpha_{t}^{(2\hk-1)}&= \sum_{m=1}^\infty I_{2m - 1}(g_{t,0} (2m - 1, 2\hk -1)).
\end{align*}
\end{theorem}

\begin{theorem}
\label{thm:convergenceL2even}
Let $k = 2k'$, where $k' \in \naturals$. Suppose $H \in (0,\frac{1}{2k'+1})$. Then $\alpha_t^{(2k')}$ exists in $L^2(\Omega)$. Moreover, $\alpha_t^{(2k')}$ exhibits the Wiener chaos expansion
\begin{align*}
	\alpha_{t}^{(2k')}&= \sum_{l=0}^\infty I_{2l}(g_{t,0} (2l, 2k')).
\end{align*}
\end{theorem}

We reiterate that the existence parts of \Cref{thm:convergenceL2odd,thm:convergenceL2even} were proved in \citep{yu2020higher} by different methods, although the formulas for the Wiener chaos expansions are new. The following result is the Varadhan-type of renormalization for the even derivatives. It is essentially a restatement of \Cref{vara}, but also includes the expression for the chaos expansion of the renormalized process.

\begin{theorem}
\label{thm:convergenceL2evenrenorm}
Let $k = 2k'$, where $k' \in \naturals$. Suppose $H \in (0,\frac{2}{4k'+1})$. Then $\alpha_t^{(2k')} - \Ex[\alpha_t^{(2k')}]$ exists in $L^2(\Omega)$. Moreover, $\alpha_t^{(2k')} - \Ex[\alpha_t^{(2k')}]$ exhibits the Wiener chaos expansion
\begin{align*}
	\alpha_{t}^{(2k')} - \Ex[\alpha_t^{(2k')}]&= \sum_{l=1}^\infty I_{2l}(g_{t,0} (2l, 2k')).
\end{align*}
\end{theorem}

This result clears up an apparent anomaly in Yu's paper \citep{yu2020higher}, namely that the formula for the critical values of $H$ were different for odd and even $k$. We note that now, after the renormalization, in both cases the critical value lies at $H = \frac{2}{2k+1}$. For self-intersection local time (SLT) in the planar case, the Varadhan renormalization can be carried out by subtracting out the part of the self-intersection in $\{0 \leq u \leq s \leq t\}$ corresponding to the spurious intersection on the set $\{u=s\}$; this is explained clearly in \citep{rosen1986tanaka} by means of a Tanaka formula. We do not know whether any such interpretation is possible for the even derivatives covered here; a possible avenue towards this would be some sort of Tanaka formula, as described at the end of \Cref{sec:literaturereview}. In the cases of the odd $k$, symmetry and the fact that $\delta^{(k)}$ is odd implies that the expectation of $\alpha^{(k)}_{t,\vep}$ is zero, and there is therefore no place for renormalization. This difference between even and odd derivatives is mentioned in \cite{rosen2005derivatives} when comparing the DSLT of Brownian motion with one derivative in one dimension against SLT in two dimensions. \\

The Wiener chaos expansions given above allow us to study regularity properties of DSLT. The space $\mathbb{D}^{1,2}$ is the space of square integrable Malliavin differentiable random variables, see \Cref{defn:Dkpspace}. One can utilize the Wiener chaos expansion along with \Cref{prop:D12space} in order to confirm whether a random variable belongs in $\mathbb{D}^{1,2}$. We will prove the following.
\begin{theorem}
\label{thm:convergenceD2}
\noindent
\begin{enumerate}[label = (\roman*), ref = \roman*]
    \item Let $k = 2\hk - 1$, where $\hk \in \naturals$. Suppose $H \in (0, \frac{2}{4\hk +1})$. Then $\alpha_{t}^{(2\hk-1)} \in \mathbb{D}^{1,2}$. \label{thm:convergenceD2odd}

    \item Let $k = 2k'$, where $k' \in \naturals$. Suppose $H \in (0, \frac{2}{4k' +3})$. Then $\alpha_{t}^{(2k')} -  \Ex[\alpha_t^{(2k')}] \in \mathbb{D}^{1,2}$. \label{thm:convergenceD2even}

\end{enumerate}
\end{theorem}

Note that, in both cases, the relevant process associated to the $k$-th derivative belongs to $\mathbb{D}^{1,2}$ whenever $H \in (0, \frac{2}{2k+3})$. We remark that similar questions for SLT were considered in \citep{hu2002chaos, hu2007regularity}.
\section{Proofs}
\label{sec:proofs}
\subsection*{Proof of \Cref{lemma:wienerchaos}}
The main tools are \Cref{thm:stroock} (Stroock's formula), as well as the Fourier identity 
\begin{align*}
	f_\vep(x) =  \frac{1}{2 \pi} \int_\reals e^{i p x} e^{-\vep p^2/2} \dd p
\end{align*}
which is a simple consequence of the characteristic function of a $\NN(0, \vep)$ random variable. This leads to the representation 
\begin{align*}
	f^{(k)}_\vep(x) = \frac{i^k}{2 \pi} \int_\reals p^k e^{i p x} e^{-\vep p^2/2} \dd p.
\end{align*}
Since 
\begin{align*}
	\alpha_{t,\vep}^{(k)} = (-1)^k \int_0^t \int_0^s f^{(k)}_\vep(B^H_s  - B^H_r) \dd r \dd s,
\end{align*}
then obtaining the Wiener chaos expansion of $f_\vep^{(k)}(B^H_s - B^H_r)$ will in turn result in the Wiener chaos expansion of $\alpha_{t,\vep}^{(k)}$. By \Cref{thm:stroock} (Stroock's formula),
\begin{align*}
		f_\vep^{(k)}(B^H_s - B^H_r) &= \sum_{n=0}^\infty I_{n}(h_\vep(n,k))
\end{align*}
where
\begin{align*}
	h_\vep(n,k) \equiv h_\vep(n,k;s,r;v_1, \dots, v_n) =  \frac{1}{n!} \Ex \big [D^{n}_{v_1, \dots , v_n} f_\vep^{(k)}(B^H_s - B^H_r)\big ].
\end{align*}
From \Cref{appen:MalliavinfBm}, we have that 
\begin{align*}
	D^{n} f_\vep^{(k)}(B^H_s - B^H_r) = f_\vep^{(k+n)}(B_s^H-B_r^H) (K_{r,s})^{\otimes n}, 
\end{align*}
where $K_{r,s}$ is defined in \cref{defnKthing}.
Thus 
\begin{align*}
 	\Ex \big [D^{n}_{v_1, \dots , v_n} f_\vep^{(k)}(B^H_s - B^H_r)\big ] = \Ex \big [f^{(k+n)}_\vep (B^H_s - B^H_r) \big ]\prod_{j=1}^n K_{r,s}(v_j).
\end{align*}
Using the Fourier identity, we have that 
\begin{align*}
	 \Ex \big [ f^{(k+n)}_\vep (B^H_s - B^H_r) \big ] &=  \frac{i^{k+n}}{2 \pi} \int_\reals p^{k+n} \Ex \big [e^{i p (B^H_s - B^H_r)}\big ] e^{-\vep p^2/2} \dd p \\
	 						&=  \frac{i^{k+n}}{2 \pi} \int_\reals p^{k+n} e^{-\rho^2 [(s - r)^{2H} + \vep]/2 }  \dd p.
\end{align*}
Changing variable $p \leftarrow p \left ( (s - r)^{2H} + \vep \right )^{1/2}$, then
\begin{align*}
	\Ex \big [ f^{(k+n)}_\vep (B^H_s - B^H_r) \big ] & = \frac{i^{k+n}}{2 \pi  [(s - r)^{2H} + \vep]^{(k+n+1)/2}} \int_\reals p^{k+n} e^{-\rho^2/2 }  \dd p.
\end{align*}
The preceding integral is $0$ if $k+n$ is odd. Recall $P$ is the set of tuples of non-negative integers whose sum is even. Then
\begin{align*}
	\Ex \big [ f^{(k+n)}_\vep (B^H_s - B^H_r) \big ] & = \frac{i^{k+n}\sqrt{2 \pi} (n+k - 1)!!}{2 \pi  [(s - r)^{2H} + \vep]^{(k+n+1)/2}}\gind{P}{n,k} \ \\
						 & = \frac{(-1)^{(k+n)/2} (n+k - 1)!!}{\sqrt{2 \pi}  [(s - r)^{2H} + \vep]^{(k+n+1)/2}}\gind{P}{n,k} .
\end{align*}
Thus 
\begin{align*}
	h_\vep(n,k) = \left (\frac{1}{n!} \prod_{j=1}^n K_{r,s}(v_j) \right ) \frac{(-1)^{(k+n)/2} (n+k - 1)!!}{\sqrt{2 \pi}  [(s - r)^{2H} + \vep]^{(k+n+1)/2}}\gind{P}{n,k} .
\end{align*}
If $k$ is odd then $k+n$ is even if and only if $n$ is odd. If $k$ is even then $k+n$ is even if and only if $n$ is even. So we have two cases, depending on the parity of $k$. Let $\hk, k' \in \naturals$. Hence,
\begin{align*}
	f_\vep^{(2\hk-1)}(B^H_s - B^H_r) &= \sum_{m=1}^\infty I_{2m-1}(h_\vep(2m-1,2\hk -1)),
\end{align*}
and
\begin{align*}
	f_\vep^{(2k')}(B^H_s - B^H_r) &= \sum_{l=0}^\infty I_{2l}(h_\vep(2l,2k')).
\end{align*}
Lastly, take 
\begin{align*}
	g_{t,\vep}(n,k) &\equiv g_{t,\vep}(n,k;v_1, \dots, v_n) \\&=  (-1)^k \int_0^t \int_0^s h_\vep(n,k) \dd r \dd s \\& =  \left (\frac{(-1)^{(3k+n)/2} (n+k - 1)!!  }{n! \sqrt{2 \pi}  }\right ) \left ( \int_0^t \int_0^s \frac{\prod_{j=1}^n K_{r,s}(v_j) }{[(s - r)^{2H} + \vep]^{(k+n+1)/2}} \dd r \dd s \right )\gind{P}{n,k}.
\end{align*}
Then the Wiener chaos expansions of $\alpha_{t,\vep}^{(2\hk-1)}$ and $\alpha_{t,\vep}^{(2k')}$ follow. \qed
\subsection*{Proof of \Cref{thm:convergenceL2odd}}
The idea is to utilize \Cref{lemma:sufficientconvergence} as well as the Wiener chaos expansion of $\alpha^{(k)}_{t,\vep}$, which is given in \Cref{lemma:wienerchaos}. Since all homogeneous chaoses are orthogonal to each other, the $L^2(\Omega)$ norm of the Wiener chaos expansion is the infinite sum of the $L^2(\Omega)$ norms of each individual chaos term. So it suffices to study the object
\begin{align*}
		\Ex \big[I_{n}(g_{t,\vep}(n,k))^2\big]
\end{align*}
for when $(n,k) \in P$. The objective then is to show that the Wiener chaos expansion is bounded uniformly with respect to $\vep$ in $L^2(\Omega)$, and then appeal to \Cref{lemma:sufficientconvergence}. Since the function $g_{t,\vep}(n,k) \equiv g_{t,\vep}(n,k; v_1, \dots , v_n)$ is symmetric in $\tilde v := (v_1, \dots ,v_n)$, then a well known identity pertaining to multiple Wiener integrals yields
\begin{align*}
		\Ex \big [I_{n}(g_{t,\vep}(n,k))^2 \big] = n! \| g_{t,\vep}(n,k) \|^2_{\boH^{\otimes n}}.
\end{align*}
Fix $(n,k) \in P$. Recall the simplex $D_t = \{(r,s) \in \reals^2 : 0 \leq r \leq  s \leq t\}$. Then
\begin{align*}
	 \| g_{t,\vep}(n,k) \|^2_{\boH^{\otimes n}} =  \frac{[(n+k - 1)!!]^2}{(n!)^2 2 \pi} \int_{D_t^2}  \frac{1}{ [(s - r)^{2H} + \vep]^{(k+n+1)/2} [(s' - r')^{2H} + \vep]^{(k+n+1)/2}} \\ 
	 \cdot \int_{[0,T]^n}   \left (\prod_{j=1}^n K_{r,s}(v_j) K_{r',s'}(v_j) \right )  \dd \tilde v \dd r \dd s \dd r' \dd s'.
\end{align*}
Since $(n,k) \in P$, then $n+k -1$ is odd, and hence $(n+k-1)!! = \frac{(n+k-1)!}{2^{(n+k-2)/2} \left(\frac{n+k-2}{2}\right)!}$. Recall
\begin{align*}
	\lambda &= \Var( B_s^H - B_r^H) = |s - r|^{2H},\\
	\rho &= \Var( B_{s'}^H - B_{r'}^H) = |s'-r'|^{2H}, \\
	\mu &= \Cov(B_s^H - B_r^H, B_{s'}^H - B_{r'}^H) = \frac{1}{2}\left (|s - r'|^{2H} + |r - s'|^{2H} - |s - s'|^{2H} - |r - r'|^{2H}\right ).
\end{align*}
Furthermore, from \Cref{appen:MalliavinfBm}, $\int_0^T K_{r,s}(u) K_{r', s'}(u) \dd u = \langle K_{r,s}, K_{r',s'} \rangle_{\boH} = \mu$, and thus 
\begin{align*}
	\int_{[0,T]^n}   \left (\prod_{j=1}^n K_{r,s}(v_j) K_{r',s'}(v_j) \right )  \dd \tilde v = \mu^n.
\end{align*}
Now take  $\vep =0$ to maximize the integrand over $D_t^2$. This yields,
\begin{align*}
	n!\| g_{t,0}(n,k) \|^2_{\boH^{\otimes n}} =  \frac{[(n+k - 1)!]^2}{n! (2 \pi)2^{n+k-2}\left [ \left ( \frac{n+k-2}{2} \right )!\right ]^2} \int_{D_t^2}  \frac{\mu^n}{ \lambda^{(k+n+1)/2} \rho^{(k+n+1)/2}} \dd r \dd s \dd r' \dd s'.
\end{align*}
Let $\gamma := \mu/\sqrt{\lambda \rho}$. Then we get 
\begin{align}
	\sum_{n=0}^{\infty} \Ex \big [I_{n}(g_{t,0}(n,k))^2 \big] = \frac{1}{2^{k-1} \pi} \int_{D_t^2} \left \{ \sum_{n=0}^{\infty} \frac{[(n+k - 1)!]^2 \gamma^n}{n!2^n\left [ \left ( \frac{n+k-2}{2} \right )!\right ]^2} \right \} \frac{1}{ \lambda^{(k+1)/2} \rho^{(k+1)/2}} \dd r \dd s \dd r' \dd s'. \label{chaosgengamma}
\end{align}
Now take $k = 2 \hk-1$, for $\hk \in \naturals$. Then the $L^2(\Omega)$ norm of $\alpha_{t,0}^{(2 \hk-1)}$ is 
\begin{align*}
	&\sum_{m=1}^{\infty} \Ex \big [I_{2m-1}(g_{t,0}(2m-1,2\hk-1))^2 \big] \\&=  \frac{1}{2^{2\hk-2} \pi} \int_{D_t^2} \left \{ \sum_{m=1}^{\infty} \frac{[(2m +2\hk - 3)!]^2 \gamma^{2m-1}}{(2m-1)!2^{2m-1}\left [ \left ( m+\hk-2 \right )!\right ]^2} \right \} \frac{1}{ \lambda^{\hk} \rho^{\hk}} \dd r \dd s \dd r' \dd s'.
\end{align*}
After a tedious rearrangement, we can write the preceding infinite sum in terms of the Hypergeometric function, specifically
\begin{align*}
	 \sum_{m=1}^{\infty} \frac{[(2m +2\hk - 3)!]^2 \gamma^{2m-1}}{(2m-1)!2^{2m-1}\left [ \left ( m+\hk-2 \right )!\right ]^2}
	 = 	 \left ( \frac{(2 \hk - 1)!}{\sqrt{2}[(\hk-1)!] }\right )^2  \gamma F_{1,2}\left (\hk+\frac{1}{2}, \hk + \frac{1}{2}; 3/2; \gamma^2\right),
\end{align*}
where $F_{1,2}$ is the Hypergeometric function 
\begin{align*}
	F_{1,2}(a,b;c;z) := \sum_{d=0}^\infty \frac{a^{\bar d } b^{\bar d}}{c^{\bar d}} \frac{z^d}{d!}
\end{align*}
and $a^{\bar d} := \prod_{j=1}^d (a+j-1)$ is the rising factorial. Then, using the Euler identity, $F_{1,2}(a,b; c; z) = (1-z)^{c-a-b}  F_{1,2}(c-a,c-b; c; z)$, as well as $(-a)^{\bar d} = 0$ for $d > a>0$ yields
\begin{align*}
F_{1,2}\left (\hk+\frac{1}{2}, \hk + \frac{1}{2}; 3/2; \gamma^2\right)
	 =  \frac{1}{(1-\gamma^2)^{(4\hk-1)/2}} \sum_{d=0}^{\hk -1} C_{d}(\hk) \gamma^{2d}
\end{align*}
where
\begin{align*}
	C_d(\hk) := \frac{ [(1-\hk)^{\bar d}]^2 }{(3/2)^{\bar d} d!}.
\end{align*}
Thus, in order to show existence of $\alpha_t^{(2 \hk - 1)}$ in $L^2(\Omega)$, it suffices to determine for which values of $\hk \in \naturals$ and $H \in (0,1)$ the integral
\begin{align}
	\int_{D_t^2}  \frac{1}{(1-\gamma^2)^{(4\hk-1)/2} \lambda^{\hk} \rho^{\hk}} \sum_{d=0}^{\hk -1}  \gamma^{2d+1}\dd r \dd s \dd r' \dd s' \label{keyintegraloddsum}
\end{align}
is finite. Notice that we can take out the sum from the integral, so that the finiteness condition of \cref{keyintegraloddsum} can be replaced with finiteness of
\begin{align*}
	\int_{D_t^2}  \frac{ \gamma^{2d+1}}{(1-\gamma^2)^{(4\hk-1)/2} \lambda^{\hk} \rho^{\hk}}\dd r \dd s \dd r' \dd s'
\end{align*}
for each $d = 0, 1, \dots, \hk -1$. Now since $|\gamma| \leq 1$, we have that $|\gamma|^{d+1} \leq |\gamma|^d$, so it suffices to determine the convergence of the preceding integral for $d = 0$, namely, the integral 
\begin{align}
	\int_{D_t^2}  \frac{\gamma}{(1-\gamma^2)^{(4\hk-1)/2} \lambda^{\hk} \rho^{\hk}}\dd r \dd s \dd r' \dd s'. \label{keyintegralodd}
\end{align}
From \Cref{lemma:finitenessL2odd}, \cref{keyintegralodd} is finite for $H \in (0, \frac{2}{4\hk - 1})$ for all $\hk \in \naturals$. The Wiener chaos expansion of $\alpha_t^{(2 \hk - 1)}$ then follows from \Cref{lemma:wienerchaos} by letting $\vep \downarrow 0$. \qed
\subsection*{Proof of \Cref{thm:convergenceL2evenrenorm}}
The first part of the proof is identical to that of the proof of \Cref{thm:convergenceL2odd}. However, for the second part, instead take $k = 2k'$, for $k' \in \naturals$. Notice also that $\Ex[\alpha_{t,0}^{(2k')}]= I_0(g_{t,0}(0,2k'))$. Leading on from \cref{chaosgengamma}, the $L^2(\Omega)$ norm of $\alpha_{t,0}^{(2k')} - \Ex[\alpha_{t,0}^{(2k')}]$ is 
\begin{align*}
	\sum_{l=1}^{\infty} \Ex \big [I_{2l}(g_{t,0}(2l,2k'))^2\big ] =  \frac{1}{2^{2k'-1} \pi} \int_{D_t^2} \left \{ \sum_{l=1}^{\infty} \frac{[(2l +2k' - 1)!]^2 \gamma^{2l}}{(2l)!2^{2l}\left [ \left ( l+k'-1 \right )!\right ]^2} \right \} \frac{1}{ \lambda^{k'+\frac{1}{2}} \rho^{k'+\frac{1}{2}}} \dd r \dd s \dd r' \dd s'.
\end{align*}
It can be seen that
\begin{align*}
	 \sum_{l=1}^{\infty} \frac{[(2l +2 k' - 1)!]^2 \gamma^{2l}}{(2l)!2^{2l}\left [ \left ( l+k'-1 \right )!\right ]^2}
	 &= 	 \left ( \frac{(2 k' - 1)!}{(k'-1)! }\right )^2 \left (-1 + F_{1,2}\left (k'+\frac{1}{2}, k' + \frac{1}{2}; 1/2; \gamma^2\right) \right ) \\
	 &= 	 \left ( \frac{(2 k' - 1)!}{(k'-1)! }\right )^2 \frac{1}{(1-\gamma^2)^{(4k' +1)/2}} \sum_{d=1}^{k'} D_d(k') \gamma^{2d},
\end{align*}
where
\begin{align*}
	D_d(k') := \frac{ [(k')^{\bar d}]^2 }{(1/2)^{\bar d} d!}.
\end{align*}
Thus, in order to show existence of $\alpha_t^{(2 k')} - \Ex [ \alpha_t^{(2 k')}] $ in $L^2(\Omega)$, it suffices to determine for which values of $k' \in \naturals$ and $H \in (0,1)$ the integral
\begin{align}
	\int_{D_t^2} \frac{1}{(1-\gamma^2)^{(4k' + 1)/2} \lambda^{k'+\frac{1}{2}} \rho^{k'+\frac{1}{2}}}  \sum_{d=1}^{k'} \gamma^{2d}  \dd r \dd s \dd r' \dd s'
\end{align}
is finite. Similar to arguments made in the proof of \Cref{thm:convergenceL2odd}, we can take out the sum from the integral and notice that it suffices to study the case of $d = 1$, namely, the integral 
\begin{align}
	\int_{D_t^2} \frac{\gamma^2}{(1-\gamma^2)^{(4k' + 1)/2} \lambda^{k'+\frac{1}{2}} \rho^{k'+\frac{1}{2}}}\dd r \dd s \dd r' \dd s'. \label{keyintegraleven}
\end{align}
From \Cref{lemma:finitenessL2even}, \cref{keyintegraleven} is finite for $H \in (0, \frac{2}{4k' + 1})$ for all $k' \in \naturals$. The Wiener chaos expansion of $\alpha_t^{(2 k')} - \Ex [ \alpha_t^{(2 k')}] $ then follows from \Cref{lemma:wienerchaos} by subtracting the $0$-th order term and taking $\vep \downarrow 0$.  
\qed
\begin{remark}
We have omitted the proof of \Cref{thm:convergenceL2even}, since it is essentially identical to the proof of \Cref{thm:convergenceL2evenrenorm}.
\end{remark}
\subsection*{Proof of \Cref{thm:convergenceD2}}
We will first consider the odd case, which is \cref{thm:convergenceD2odd}. Let $k = 2\hk - 1$ for $\hk \in \naturals$. Appealing to the Wiener chaos expansion from \Cref{thm:convergenceL2odd} as well as \Cref{prop:D12space}, the goal is to study the finiteness of
\begin{align*}
	 \sum_{m=1}^{\infty} 2m \Ex \big [I_{2m-1}(g_{t,0}(2m-1,2\hk-1))^2\big].
\end{align*}
The following calculations are very similar to those from the proof of \Cref{thm:convergenceL2odd}, and so we will omit some details. Now we have
\begin{align*}
	 &\sum_{m=1}^{\infty} 2m \Ex \big [I_{2m-1}(g_{t,0}(2m-1,2\hk-1))^2 \big] \\
	 &=\frac{1}{2^{2\hk-2} \pi} \int_{D_t^2} \left \{ \sum_{m=1}^{\infty} \frac{[(2m +2\hk - 3)!]^2 2m \gamma^{2m-1}}{(2m-1)!2^{2m-1}\left [ \left ( m+\hk-2 \right )!\right ]^2} \right \} \frac{1}{ \lambda^{\hk} \rho^{\hk}} \dd r \dd s \dd r' \dd s'\\
	 &=\frac{1}{2^{2\hk-2} \pi} \int_{D_t^2} \left \{ \frac{\dd}{\dd \gamma} \left (\gamma \sum_{m=1}^{\infty} \frac{[(2m +2\hk - 3)!]^2 \gamma^{2m-1}}{(2m-1)!2^{2m-1}\left [ \left ( m+\hk-2 \right )!\right ]^2} \right )\right \} \frac{1}{ \lambda^{\hk} \rho^{\hk}} \dd r \dd s \dd r' \dd s'.
\end{align*}
This preceding infinite sum is
\begin{align*}
	    \sum_{m=1}^{\infty} \frac{[(2m +2\hk - 3)!]^2 \gamma^{2m-1}}{(2m-1)!2^{2m-1}\left [ \left ( m+\hk-2 \right )!\right ]^2} &= \left ( \frac{(2 \hk - 1)!}{\sqrt{2}[(\hk-1)!] }\right )^2  \gamma F_{1,2}\left (\hk+\frac{1}{2}, \hk + \frac{1}{2}; 3/2; \gamma^2\right) \\
	    &= \left ( \frac{(2 \hk - 1)!}{\sqrt{2}[(\hk-1)!] }\right )^2  \frac{\gamma}{(1-\gamma^2)^{(4\hk-1)/2}} \sum_{d=0}^{\hk -1} C_{d}(\hk) \gamma^{2d}.
\end{align*}
So it suffices to study convergence of the following integral 
\begin{align*}
	 \int_{D_t^2} \left \{  \frac{\dd}{\dd \gamma} \left ( \frac{1}{(1-\gamma^2)^{(4\hk-1)/2}}  \sum_{d=0}^{\hk -1} C_{d}(\hk) \gamma^{2d+2} \right ) \right \} \frac{1}{ \lambda^{\hk} \rho^{\hk}} \dd r \dd s \dd r' \dd s'.
\end{align*}
Doing the differentiation and ignoring constants, this becomes
\begin{align*}
	 \int_{D_t^2} \frac{1}{(1-\gamma^2)^{(4\hk-1)/2} \lambda^{\hk} \rho^{\hk}}  \sum_{d=0}^{\hk -1} \gamma^{2d+1}\dd r \dd s \dd r' \dd s' +  \int_{D_t^2} \frac{1}{(1-\gamma^2)^{(4\hk+1)/2} \lambda^{\hk} \rho^{\hk}} \sum_{d=0}^{\hk -1} \gamma^{2d+3} \dd r \dd s \dd r' \dd s'.
\end{align*}
Notice that the first integral here corresponds to the integral studied in \Cref{lemma:finitenessL2odd}, and thus converges for $H \in (0, \frac{2}{4\hk - 1})$. Furthermore, since $|\gamma| \leq 1$, we have that $|\gamma|^{d+1} \leq |\gamma|^d$. So it suffices to study the convergence of the second integral with $d = 0$, namely
\begin{align}
	 \int_{D_t^2} \frac{\gamma^3}{(1-\gamma^2)^{(4\hk+1)/2} \lambda^{\hk} \rho^{\hk}} \dd r \dd s \dd r' \dd s'. \label{keyintegraloddsmooth}
\end{align}
We can rewrite the integrand as 
\begin{align*}
	 \frac{\gamma^3}{(1-\gamma^2)^{(4\hk+1)/2} \lambda^{\hk} \rho^{\hk}} = \frac{\mu^3 \lambda^{\hk-1} \rho^{\hk-1}}{(\lambda \rho - \mu^2)^{(4\hk +1)/2}}.
\end{align*}

Since $\mu$ is bounded, we have $|\mu^3| \leq C \mu^2$, and we may therefore prove what we require by showing that

\begin{align*}
	\int_{ D_t^2} \frac{\mu^{2} \lambda^{\hk-1}\rho^{\hk-1}}{(\lambda \rho - \mu^2)^{(4\hk+1)/2}} \dd r \dd s \dd r' \dd s' <\infty,
\end{align*}

which one will notice is the key integral considered in \Cref{lemma:finitenessL2even}. It is therefore finite for $H \in (0, \frac{2}{4\hk+1})$, as is shown in the lemma. This handles the odd case. \\

Now we turn to the even case, which is \cref{thm:convergenceD2even}. Let $k = 2k'$ where $k'\in \naturals$. Utilizing the Wiener chaos expansion from \Cref{thm:convergenceL2evenrenorm} as well as \Cref{prop:D12space}, this implies that we need to show that 
\begin{align*}
	\sum_{l=1}^\infty (2l+1) \Ex[I_{2l}(g_{t,0}(2l,2k'))^2]
\end{align*}
is finite for $H \in (0, \frac{2}{4k'+3})$. Using similar calculations to the proof of \cref{thm:convergenceD2odd}, we will end up requiring finiteness of the integral 
\begin{align*}
	 \int_{D_t^2} \left \{  \frac{\dd}{\dd \gamma} \left ( \frac{1}{(1-\gamma^2)^{(4k'+1)/2}}  \sum_{d=1}^{k'} \gamma^{2d+1} \right ) \right \} \frac{1}{ \lambda^{k'+\frac{1}{2}} \rho^{k'+\frac{1}{2}}} \dd r \dd s \dd r' \dd s'.
\end{align*}

Doing the differentiation, then realising that one of the two resulting integrals corresponds to the key integral from \Cref{lemma:finitenessL2even}, we will then require showing finiteness of 

\begin{align} \label{stevie}
	 \int_{D_t^2} \frac{\gamma^4}{(1-\gamma^2)^{(4k'+3)/2} \lambda^{k'+\frac{1}{2}} \rho^{k'+\frac{1}{2}}} \dd r \dd s \dd r' \dd s'.
\end{align}

Notice that the integrand can be reexpressed as

\begin{align*}
	 \frac{\gamma^4}{(1-\gamma^2)^{(4k'+3)/2} \lambda^{k'+\frac{1}{2}} \rho^{k'+\frac{1}{2}}} = \frac{\mu^4 \lambda^{k'-1} \rho^{k'-1}}{(\lambda \rho - \mu^2)^{(4k' +3)/2}}.
\end{align*}

Since $\mu, \rho, \lambda$ are all bounded, we see $\mu^4 \lambda^{k'-1} \rho^{k'-1} \leq C |\mu|\lambda^{(k'+1)-1} \rho^{(k'+1)-1}$, and we may bound \cref{stevie} by a constant times

\begin{align*} 
	 \int_{D_t^2} \frac{|\mu|\lambda^{(k'+1)-1} \rho^{(k'+1)-1}}{(\lambda \rho - \mu^2)^{(4(k'+1) -1)/2}} \dd r \dd s \dd r' \dd s'.
\end{align*}

We now recognize this as the key integral considered in \Cref{lemma:finitenessL2odd}, with $k'+1$ in place of $\hk$. It is therefore finite, as is shown there, when $H < \frac{2}{4(k'+1)-1} = \frac{2}{4k'+3}$. This handles the even case, and completes the proof of the theorem. \qed

\begin{remark}
It would be reasonable to ask whether the bounds on $H$ in \Cref{thm:convergenceD2} can be improved. After all, we replaced $|\mu^3|$ by $\mu^2$ in the numerator of the integrand for the proof of \cref{thm:convergenceD2odd} (odd case), and similarly replaced $\mu^4 \lambda^{k'-1} \rho^{k'-1}$ by $|\mu|\lambda^{(k'+1)-1} \rho^{(k'+1)-1}$ in the proof of \cref{thm:convergenceD2even} (even case). However, it is not hard to see that these changes did not come at the cost of convergence in our arguments, at least not without additional ideas introduced. To see this, suppose in the proof of \cref{thm:convergenceD2odd} that we did not replace $|\mu^3|$ by $\mu^2$. Then to prove finiteness of the corresponding integral, consider implementing the strategy from \textbf{Case 1} in the proof of \Cref{lemma:finitenessL2even}. Following the strategy, one of the terms would result in an integrand with $b^{6H}$ in the numerator. However $H$ will be restricted by the exponents of $a$ and $c$ in the denominator of the integrand. So the $b^{6H}$ term in the numerator does not increase the range of convergence, and thus one can safely replace $|\mu|^3$ by $\mu^2$, in which case \Cref{lemma:finitenessL2even} is directly applicable. A similar argument can be made regarding the proof of \cref{thm:convergenceD2even}, by considering the strategy of \textbf{Case 1} in the proof of \Cref{lemma:finitenessL2odd}, with a slight difference being that our replacement modifies the powers of $\lambda$ and $\rho$, which will in turn change the powers of $(a+b)$ and $(b+c)$ in the integrand. Nonetheless, the situation is the same as far as convergence is concerned, as $H$ is still restricted by the exponents of $a$ and $c$ in the denominator of the integrand.
\end{remark}

\section*{Acknowledgements}
We would like to thank Maher Boudabra and Yidong Shen for valuable conversations. We would also like to thank an anonymous reviewer for helpful comments.
{\onehalfspacing
\renewcommand{\bibname}{References}
\bibliographystyle{unsrt}
\bibliography{biblio}
}
\appendix
\section{Malliavin calculus background}
\label{appen:malliavin}
In this appendix, we provide a short discussion on the Malliavin calculus machinery utilized in this paper. The content here is mainly inspired by \citep{nualart2006malliavin}, and thus we will omit any proofs. Let $W = \{W(\varphi): \varphi \in \HH\}$ be a collection of random variables, where $\HH$ is a real, separable Hilbert space. $W$ is called an isonormal Gaussian process if for any $\varphi \in \HH$, the random variable $W(\varphi)$ is normal with zero mean and for any $\varphi^1, \varphi^2 \in \HH$, we have
\begin{align*}
	\Ex[W(\varphi^1) W( \varphi^2)] = \langle \varphi^1, \varphi^2 \rangle_\HH.
\end{align*}
\subsection{Malliavin differentiation and Wiener chaos}
\begin{definition}[Malliavin derivative]
Let $W$ be an isonormal Gaussian process with respect to an underlying Hilbert space $\HH$. Denote by $C^{\infty}_b(\reals^n)$ the space of bounded smooth functions on $\reals^n$. Consider the space of random variables 
\begin{align*}
	\mathcal{S} := \{F = f(W(h_1), \dots, W(h_n)), f \in C^{\infty}_b(\reals^n), h_i \in \HH \text{ for each } i =1, \dots, n\}.
\end{align*}
The Malliavin derivative of $F \in \mathcal{S}$, denoted by $DF$, is given by\footnote{Here, $\partial_i f(x_1,\dots, x_i, \dots, x_n) \equiv \frac{\partial}{\partial x_i}f(x_1,\dots, x_i, \dots, x_n)$.}
\begin{align*}
	DF = \sum_{i=1}^n \partial_i f(W(h_1), \dots, W(h_n)) h_i.
\end{align*}
The $k$-fold iterated Malliavin derivative of $F \in \mathcal{S}$, denoted by $D^k F$, is given by
\begin{align*}
	D^kF = \sum_{i_1, i_2, \dots, i_k = 1}^n \partial^{k}_{i_1, i_2, \dots, i_k} f(W(h_1), \dots, W(h_n)) \bigotimes_{j=1}^k h_{i_j}.
\end{align*}
In particular, for $n = 1$, then $DF = f'(W(h)) h$ and $D^kF = f^{(k)}(W(h)) h^{\otimes k}$.
\end{definition}
\begin{definition}
\label{defn:Dkpspace}
Denote by $\mathbb{D}^{k,p}$ the completion of $\mathcal{S}$ with respect to the seminorm 
\begin{align*}
	\| F \|_{k,p} := \left [ \Ex|F|^p + \sum_{j=1}^k \Ex[\|D^j F \|^p_{\HH^{\otimes j}}] \right ]^{1/p}.
\end{align*}
The domain of the Malliavin derivative $D^k$ in $L^p(\Omega)$ is $\mathbb{D}^{k,p}$. 
\end{definition}
\begin{theorem}[Wiener chaos expansion]
\label{thm:wienerchaosexp}
Let $\GG$ be the $\sigma$-field generated by an isonormal Gaussian process $W$ with the underlying Hilbert space $\HH = L^2([a,b])$. Let $\xi \in L^2(\Omega)$. Then there exists functions $f_n : [a,b]^n \to \reals$, such that $\xi$ admits the representation
\begin{align*}
	\xi = \sum_{n=0}^\infty I_n(f_n), 
\end{align*}
which is called the Wiener chaos expansion of $\xi$.
\end{theorem}
\begin{proposition}
\label{prop:D12space}
Suppose $\xi \in L^2(\Omega)$ and possesses Wiener chaos expansion as in \Cref{thm:wienerchaosexp}. Then $\xi \in \mathbb{D}^{1,2}$ if and only if
\begin{align*}
	 \sum_{n=0}^\infty (n+1) \Ex[I_n(f_n)^2] < \infty.
\end{align*}
\end{proposition}
\begin{remark}
Consider the Malliavin derivative $DF$ when $n = 1$. When the underlying Hilbert space is $\HH = L^2([a,b])$, we will write 
\begin{align*}
	D_vF = f'(W(h)) h(v)
\end{align*}
and
\begin{align*}
	D^k_{v_1,v_2, \dots, v_k} F = f^{(k)}(W(h)) \prod_{j=1}^k h(v_j).
\end{align*}
\end{remark}
\begin{theorem}[Stroock's formula]
\label{thm:stroock}
Let $\xi$ possess a Wiener chaos expansion as in \Cref{thm:wienerchaosexp}. Then Stroock's formula states that the functions $f_n: [a,b]^n \to \reals$ have the explicit representation 
\begin{align*}
	f_n= \frac{1}{n!}\Ex \left [ D^{n}\xi \right ].
\end{align*}
\end{theorem}
\begin{lemma}
\label{lemma:sufficientconvergence}
Let $F_\vep$ be a collection of $L^2(\Omega)$ random variables with Wiener chaos expansions $F_\vep = \sum_{n=0}^\infty I_n(f_n^\vep)$. Let $\boH := L^2([0,T])$. Suppose that for each $n$, $f_n^\vep$ converges in $\boH^{\otimes n}$ to $f_n$ as $\vep \downarrow 0$, and that
\begin{align*}
	\sum_{n=0}^\infty \sup_\vep \Ex \big [I_n(f_n^\vep)^2 \big ] = \sum_{n=0}^\infty \sup_\vep \{ n! \| \hat f_n^\vep \|^2_{\boH^{\otimes n}}\} < \infty
\end{align*}
where $\hat f_n^\vep(v_1, \dots, v_n) := \frac{1}{n!}\sum_{\sigma \in S_n} f_n^\vep(v_{\sigma_1}, \dots, v_{\sigma_n})$ is the symmetric version of $f^\vep_n$. Then $F_\vep$ converges in $L^2(\Omega)$ to $F = \sum_{n=0}^\infty I_n(f_n)$ as $\vep \downarrow 0$.
\end{lemma}
\subsection{Malliavin calculus relating to fBm}
\label{appen:MalliavinfBm}
Let $(B_t)$ be an ordinary Brownian motion, and let $\boH := L^2([0,T])$. Then the collection $W^B := \{ W^B(\phi) := \int_0^T \phi_t \dd B_t : \phi \in \boH \}$ is an isonormal Gaussian process with respect to the underlying Hilbert space $\boH$. Let $\EE$ denote space of the indicator functions on $[0,T]$. For $u \leq t$ and $H>1/2$, define
\begin{align*}
	K_{0,t}(u) := c_H u^{\frac{1}{2} - H} \int_u^t (r - u)^{H-\frac{3}{2}}r^{H-\frac{1}{2}}\dd r,
\end{align*}
where 
\begin{align*}
	c_H := \left [ \frac{H(2H-1)}{\beta(2-2H, H-\frac{1}{2})}\right ]^{1/2},
\end{align*}
and for $H <1/2$,
\begin{align*}
	K_{0,t}(u) := c_H \Bigg [\left (\frac{t}{u} \right )^{H-\frac{1}{2}}(t-u)^{H-\frac{1}{2}}- \left (H-\frac{1}{2} \right)u^{\frac{1}{2}-H}\int_u^t r^{H-\frac{3}{2}}(r-u)^{H-\frac{1}{2}} \dd r \Bigg ],
\end{align*}
where 
\begin{align*}
	c_H := \left [ \frac{2H}{(1-2H)\beta(1-2H, H+\frac{1}{2})}\right ]^{1/2}.
\end{align*}
Here, $\beta(\cdot, \cdot)$ is the Beta function. Define the operator $M^H: \EE \to \boH$ such that  
\begin{align}
	M^H \rind{[s,t]}(u) := K_{s,t}(u):=  K_{0,t}(u)\gind{[0,t]}{u} - K_{0,s}(u)\gind{[0,s]}{u}.\label{defnKthing}
\end{align}
It can be shown that 
\begin{align*}
	\langle K_{0,t}(\cdot) \rind{[0,t]}(\cdot), K_{0,s}(\cdot) \rind{[0,s]}(\cdot) \rangle_{\boH} = R_H(t,s).
\end{align*}
Define $\mathbb{H}$ to be the Hilbert space given as the completion of $\EE$ with respect to the inner product 
\begin{align*}
	\langle \rind{[0,s]}, \rind{[0,t]} \rangle_{\mathbb{H}} := \langle K_{0,t}(\cdot) \rind{[0,t]}(\cdot), K_{0,s}(\cdot) \rind{[0,s]}(\cdot) \rangle_{\boH} = R_H(t,s).
\end{align*}
The operator $M^H$ is clearly a linear isometry between $\EE$ and $\boH$, and can be extended to map $\mathbb{H}$ to $\boH$. 
\begin{remark}
\label{remark:mu}
It is clear that 
\begin{align*}
	\int_0^T  K_{r,s}(u) K_{r',s'}(u) \dd u &=  \langle K_{r,s}, K_{r',s'} \rangle_{\boH}\\
	&= \left \langle K_{0,s}(\cdot)\gind{[0,s]}{\cdot} - K_{0,r}(\cdot)\gind{[0,r]}{\cdot}, K_{0,s'}(\cdot)\gind{[0,s']}{\cdot} - K_{0,r'}(\cdot)\gind{[0,r']}{\cdot} \right \rangle_{\boH} \\
	&= R_H(s,s') + R_H(r,r') - R_H(s,r') - R_H(r,s') \\
	&= \Cov(B_s^H - B_r^H, B_{s'}^H - B_{r'}^H) \\
	&= \mu.
\end{align*}
\end{remark}
It is well known that the process $(B_t^H)$ defined as 
\begin{align}
	B_t^H := \int_0^t K_{0,t}(u) \dd B_u \label{fBmrep}
\end{align}
is an fBm with Hurst parameter $H$. Consider the collection of random variables $W^H := \{ W^H(\varphi) := \int_0^T M^H \varphi(u) \dd B_u: \varphi \in \mathbb{H} \}$. Then $W^H$ is an isonormal Gaussian process with respect to the underlying Hilbert space $\mathbb{H}$. Let $f \in C_b^{\infty}(\reals)$ and define the random variable $F = f(W^H(\varphi))$. Now since for any $\varphi \in \mathbb{H}$, we have $M^H\varphi \in \boH$, then this implies $W^H(\varphi) = W^B(M^H \varphi) \in W$. So the Malliavin derivative of $F$ with respect to the isonormal Gaussian process $W^B$ is 
\begin{align*}
	DF =  f'(W^H(\varphi)) M^H \varphi, 
\end{align*} 
and the $k$-fold Malliavin derivative is 
\begin{align*}
	D^kF =  f^{(k)}(W^H(\varphi)) (M^H \varphi)^{\otimes k}.
\end{align*} 
In particular, when $\varphi = \rind{[s,t]}$, then $DF =  f'(B_t^H - B_s^H) K_{s,t}$, and $D^k F = f^{(k)}(B_t^H - B_s^H) (K_{s,t})^{\otimes k}$.
\begin{remark}
Since $W^H$ is an isonormal Gaussian process with respect to the underlying Hilbert space $\mathbb{H}$, then it is of course possible to apply the Malliavin derivative to $W^H$-measurable random variables with respect to $W^H$ itself. So the Malliavin derivative of $F = f(W^H(\varphi))$ with respect to $W^H$, denoted by $D^H F$, is $D^H F = f'(W^H(\varphi)) \varphi$. However, for our purposes, it will be simpler to compute the Malliavin derivative with respect to $W^B$ instead, since the Wiener chaos machinery utilizes multiple Wiener integrals with respect to ordinary Brownian motion.
\end{remark}
\section{Finiteness of key integrals}
\label{appen:finitenessintegrals}
In this appendix, we prove the finiteness of key integrals encountered in \Cref{thm:convergenceL2odd} and \Cref{thm:convergenceL2evenrenorm} with restrictions imposed on the Hurst parameter $H$ and order of derivative $k$. Before we proceed, we first present some necessary preliminary information. First, recall the region $D_t = \{(r,s) \in \reals^2: 0 \leq r \leq s \leq t \}$. It is then true that the region $D_t^2$ can be partitioned into three regions, given by 
\begin{align}
\begin{split}
	D^2_{1,t} &:= \{(r,s,r',s') \in D_t^2 : r < r' < s < s' \}, \label{partitionregions}\\
	D^2_{2,t} &:=  \{(r,s,r',s') \in D_t^2 : r < r' < s' < s \}, \\
	D^2_{3,t} &:= \{(r, s, r', s') \in D_t^2 : r < s < r' < s' \}.
\end{split}
\end{align}
Recall the notation 
\begin{align*}
	\lambda &=  \Var(B_s^H - B_r^H) = |s - r|^{2H},\\
	\rho &= \Var(B^H_{s'} - B^H_{r'}) =  |s'-r'|^{2H}, \\
	\mu &=\Cov(B_s^H - B_r^H, B^H_{s'} - B^H_{r'}) =  \frac{1}{2}\left (|s - r'|^{2H} + |r - s'|^{2H} - |s - s'|^{2H} - |r - r'|^{2H}\right ).
\end{align*}
The following lemma from \citep{hu2001self} (see also \citep{jung2014tanaka}) will be essential. 
\begin{lemma}
\label{lemma:boundingregions}
\pdfbookmark[2]{\Cref{lemma:boundingregions}}{lemma:boundingregions}
\noindent
\begin{enumerate}[label = (\roman*), ref = \roman*]
\item \label{lemma:boundingregionsi}Let $(r,s,r',s') \in D^2_{1,t}$ (that is, $r < r' < s < s'$). Define $a = r' - r$, $b = s -r'$ and $c = s' - s$. Then 
\begin{align*}
	\lambda \rho - \mu^2 \geq K ((a+b)^{2H}c^{2H} + a^{2H}(b+c)^{2H}).
\end{align*}
\item \label{lemma:boundingregionsii}Let $(r,s,r',s') \in D^2_{2,t}$ (that is, $r < r' < s' < s$). Define $a = r' - r$, $b = s' - r'$ and $c = s - s'$. Then 
\begin{align*}
	\lambda \rho - \mu^2 \geq Kb^{2H} (a^{2H}+ c^{2H}).
\end{align*}
\item \label{lemma:boundingregionsiii}Let $(r,s,r',s') \in D^2_{3,t}$ (that is, $r < s < r' < s'$). Define $a = s -r$, $b = r' - s$ and $c = s' - r'$. Then
\begin{align*}
	\lambda \rho - \mu^2 \geq K (a^{2H} c^{2H}).
\end{align*}
\end{enumerate}
\end{lemma}
\begin{remark}[Standard Inequalities]
We will make use of the following standard inequalities from analysis. Let $K, K_1$ and $K_2$ be a positive constants that may change from line to line.
\begin{framed}
\begin{enumerate}[label = (\roman*)]
\item $xy \leq \frac{x^p}{p} + \frac{y^q}{q} \leq K (x^p + y^q)$, where $p, q > 0$ and $q = \frac{p-1}{p}$ (Young's inequality).
\item $x^{\alpha}y^{\beta} \leq \alpha x + \beta y \leq K (x + y)$, where $\alpha, \beta > 0$ and $\alpha + \beta  = 1$ (Young's inequality variation).
\item $K_1(x + y)^\xi \leq (x^\xi+y^\xi) \leq K_2(x + y)^\xi$ for $\xi \in \reals$.
\item $(x + y)^\xi \leq (x^\xi+y^\xi)$, where $\xi \in (0,1)$.
\end{enumerate}
\end{framed}
\end{remark}
\begin{lemma}
\label{lemma:finitenessL2odd}
\pdfbookmark[2]{\Cref{lemma:finitenessL2odd}}{lemma:finitenessL2odd}
Consider \cref{keyintegralodd}, the key integral from the proof of \Cref{thm:convergenceL2odd}. Namely, 
\begin{align}
	\int_{D_t^2}  \frac{\gamma}{(1-\gamma^2)^{(4\hk-1)/2} \lambda^{\hk} \rho^{\hk}} \dd r \dd s \dd r' \dd s' \label{keyintegraloddappen}
\end{align}
where we recall $\gamma = \mu/(\sqrt{\lambda \rho})$ and $\hk \in \naturals$. Then this integral is finite if $H \in (0, \frac{2}{4 \hk - 1})$, where $\hk \in \naturals$. 
\end{lemma}
\begin{proof}
First, notice that
\begin{align*}
	  \frac{ \gamma}{(1-\gamma^2)^{(4\hk-1)/2} \lambda^{\hk} \rho^{\hk}} = \frac{\mu \lambda^{\hk  - 1} \rho^{\hk -1}}{(\lambda \rho - \mu^2)^{(4 \hk - 1)/2} }.
\end{align*}
Moreover, utilizing that the region $D_t^2$ can be partitioned into $D^2_{1,t}, D^2_{2,t}$ and $D^2_{3,t}$ given in \cref{partitionregions}, then finiteness of the integral \cref{keyintegraloddappen} is guaranteed if and only if
\begin{align*}
	\int_{ D_{i,t}^2} \frac{\mu \lambda^{\hk - 1} \rho^{\hk - 1}}{(\lambda \rho - \mu^2)^{(4 \hk - 1)/2}}   \dd r \dd s \dd r' \dd s'
\end{align*}
is finite for each $i =1, 2,3$. Define $\bH(\hk) := \frac{2}{4\hk - 1}$, which is the proposed upper bound on $H$.  In what follows, we will make use of these facts:
\begin{itemize}
\item We will denote $K$ to be a strictly positive constant whose value may change line by line.
\item $\bH(\hk)$ is strictly decreasing and its maximum is $\bH(1) = 2/3$.
\item We will often need to divide the cases into $H \geq 1/2$ and $H < 1/2$. Notice that $\bH(\hk) \geq 1/2$ if and only if $\hk = 1$.
\end{itemize}
We now proceed case by case.
\begin{enumerate}[label = \textbf{Case \arabic*}:]
\item We start by considering the integration over $D^2_{1,t}$. Following the notation in \Cref{lemma:boundingregions} \cref{lemma:boundingregionsi}, we have that $a+b = s - r $ and $b + c = s' - r'$. This gives that $\lambda = (a+b)^{2H}$ and $\rho = (b+c)^{2H}$. Now utilizing the lemma, we have 
\begin{align}
	\int_{ D_{1,t}^2} \frac{\mu \lambda^{\hk  - 1} \rho^{\hk - 1}}{(\lambda \rho - \mu^2)^{(4 \hk - 1)/2}} \dd r \dd s \dd r' \dd s' 
	&\leq  K \int_{[0,t]^3} \frac{|\mu| (a+b)^{2H(\hk - 1)} (b+c)^{2H(\hk - 1)}}{((a+b)^{2H}c^{2H} + a^{2H}(b+c)^{2H})^{(4 \hk - 1)/2}}  \dd a \dd b \dd c \nonumber\\
	&\leq  K \int_{[0,t]^3} \frac{|\mu| (a+b)^{2H(\hk - 1)} (b+c)^{2H(\hk - 1)}}{((a+b)c)^{\frac{H}{2} (4\hk  - 1)}(a(b+c))^{\frac{H}{2} (4 \hk - 1)}} \dd a \dd b \dd c  \label{youngscase1} \\
		&=  K \int_{[0,t]^3} \frac{|\mu|}{(a+b)^{\frac{3H}{2}}(b+c)^{\frac{3H}{2}} a^{\frac{H}{2} (4\hk - 1)}c^{\frac{H}{2}(4\hk- 1)}} \dd a \dd b \dd c \label{keyboundcase1}
\end{align}
where we used Young's inequality, $x^{1/2}y^{1/2} \leq \frac{x}{2} + \frac{y}{2}$ in \cref{youngscase1}. Now
\begin{align*}
	2 \mu &= \big ((a+b+c)^{2H} +  b^{2H} - a^{2H} - c^{2H} \big )
\end{align*}
which gives the trivial bound
\begin{align*}
		|\mu| &\leq K \big( a^{2H} + b^{2H} + c^{2H} \big).
\end{align*}
Thus we need to show that 
\begin{align*}
	K \int_{[0,t]^3} \frac{a^{2H} + b^{2H} + c^{2H}}{(a+b)^{\frac{3H}{2}}(b+c)^{\frac{3H}{2}} a^{\frac{H}{2} (4\hk - 1)}c^{\frac{H}{2}(4\hk- 1)}} \dd a \dd b \dd c
\end{align*}
is finite. So we look at each of the three terms involving $a^{2H}, b^{2H}$ and $c^{2H}$ in the numerator separately. First start with the $b^{2H}$ term. Then,
\begin{align*}
	&K \int_{[0,t]^3} \frac{ b^{2H}}{(a+b)^{\frac{3H}{2}} (b+c)^{\frac{3H}{2}} a^{\frac{H}{2} (4 \hk - 1)}c^{\frac{H}{2} (4 \hk - 1)}} \dd a \dd b \dd c \\
	&\leq K \int_{[0,t]^3} \frac{b^{2H}}{b^{\frac{3H}{2}} b^{\frac{3H}{2}} a^{\frac{H}{2} (4 \hk - 1)}c^{\frac{H}{2} (4 \hk - 1)}} \dd a \dd b \dd c \\
	&= K \int_{[0,t]^3} \frac{1}{b^{H} a^{\frac{H}{2} (4 \hk - 1)}c^{\frac{H}{2} (4 \hk - 1)}} \dd a \dd b \dd c
\end{align*}
which is clearly finite for $H< \bH(\hk)$. For the $a^{2H}$ term we have  
\begin{align*}
	&K \int_{[0,t]^3} \frac{ a^{2H}}{(a+b)^{\frac{3H}{2}} (b+c)^{\frac{3H}{2}} a^{\frac{H}{2} (4 \hk - 1)}c^{\frac{H}{2} (4 \hk - 1)}} \dd a \dd b \dd c \\
	&= K \int_{[0,t]^3} \frac{1}{(a+b)^{\frac{3H}{2}} (b+c)^{\frac{3H}{2}}a^{\frac{H}{2} (4 \hk - 5)}c^{\frac{H}{2} (4 \hk - 1)}} \dd a \dd b \dd c \\
	& \leq K \int_{[0,t]^3} \frac{1}{a^{\frac{3H}{2}} b^{\frac{3H}{2}}a^{\frac{H}{2} (4 \hk - 5)}c^{\frac{H}{2} (4 \hk - 1)}} \dd a \dd b \dd c \\
	&= K \int_{[0,t]^3} \frac{1}{b^{\frac{3H}{2}} a^{\frac{H}{2} (4 \hk - 2)}c^{\frac{H}{2} (4 \hk - 1)}} \dd a \dd b \dd c.
\end{align*}
Clearly the $c$ term exponent is less than 1 for $H < \bH(\hk)$. Looking at the $b$ term exponent, we have that $H < \bH(\hk) < \bH(1) = 2/3$, and so the $b$ term exponent is less than 1. Observing the $a$ term exponent, we have 
\begin{align*}
	\frac{H}{2} (4 \hk - 2) < \frac{\bH(\hk)}{2} (4 \hk - 2) = \frac{4 \hk - 2}{4\hk - 1} < 1
\end{align*}
for any $\hk \geq 1$. So this integral is finite. For the $c^{2H}$ term, notice that there is symmetry between $a$ and $c$ in the integrand of \cref{youngscase1}. Thus, this case is identical to the case of $a^{2H}$. Hence, we have completed \textbf{Case 1}. 
\item Next we look at the integration over $D^2_{2,t}$. Following the notation in \Cref{lemma:boundingregions} \cref{lemma:boundingregionsii}, we have that $a+b +c = s - r $ and $b = s' - r'$. This gives that $\lambda = (a+b+c)^{2H}$ and $\rho = b^{2H}$. Now utilizing the lemma, we have 
\begin{align}
	\int_{ D_{2,t}^2} \frac{|\mu| \lambda^{\hk  - 1} \rho^{\hk - 1}}{(\lambda \rho - \mu^2)^{(4 \hk - 1)/2}} \dd r \dd s \dd r' \dd s' \nonumber
	&\leq  K \int_{[0,t]^3} \frac{|\mu| (a+b+c)^{2H(\hk - 1)} b^{2H(\hk - 1)}}{(b^{2H} (a^{2H}+ c^{2H}))^{(4 \hk - 1)/2}}  \dd a \dd b \dd c \\
	&= K \int_{[0,t]^3} \frac{|\mu| (a+b+c)^{2H(\hk - 1)}}{b^{H(2\hk + 1)} (a^{2H}+ c^{2H})^{(4 \hk - 1)/2}}  \dd a \dd b \dd c  \nonumber \\
	&\leq  K \int_{[0,t]^3} \frac{|\mu| (a+b+c)^{2H(\hk - 1)}}{b^{H(2\hk + 1)} (a+c)^{H(4 \hk - 1)}}  \dd a \dd b \dd c \label{keyboundcase2}
\end{align}
where we have used $K_1(x^p + y^p) \leq (x+y)^p \leq K_2(x^p + y^p)$ in \cref{keyboundcase2}. Now we have 
\begin{align*}
	2 \mu = (a+b)^{2H} - a^{2H} + (b+c)^{2H} - c^{2H},
\end{align*}
which by \citep{hu2001self}, can be rewritten as 
\begin{align}
	2 \mu = 2Hb \int_0^1 \big ( (a+bu)^{2H-1} + (c+bu)^{2H-1} \big ) \dd u.  \label{intrepcase2}
\end{align}
By this integral representation, it is clear that $\mu$ is non-negative. So we bracket in the following way.
\begin{align*}
	2\mu = \left (  (a+b)^{2H} - a^{2H} \right ) + \left (  (b+c)^{2H} - c^{2H}\right ).
\end{align*}
Now for $H < 1/2$, we have $(a+b)^{2H} \leq \big (a^{2H} + b^{2H} \big )$, and similar for the other term, since the exponents are between $0$ and $1$. This gives $\mu \leq K b^{2H}$. Now for $H \geq 1/2$, following the integral representation \cref{intrepcase2}, notice that the exponents within the integrand are non-negative. Thus we can bound the integral by a constant. Succinctly
\begin{align*}
\mu \leq
	\begin{cases}
		Kb^{2H}, &H < 1/2, \\[.3cm]
		Kb, &H \geq 1/2.
	\end{cases}
\end{align*}
First consider $H < 1/2$. Leading on from \cref{keyboundcase2}, we have
\begin{align*}
	 &K \int_{[0,t]^3} \frac{\mu (a+b+c)^{2H(\hk - 1)}}{b^{H(2\hk + 1)} (a+c)^{H(4 \hk - 1)}}  \dd a \dd b \dd c \\
	 &\leq K\int_{[0,t]^3} \frac{b^{2H}(a+b+c)^{2H(\hk - 1)}}{b^{H(2\hk + 1)} (a+c)^{H(4 \hk - 1)}}  \dd a \dd b \dd c \\ 
	  &\leq K\int_{[0,t]^3} \frac{(a+b+c)^{2H(\hk - 1)}}{b^{H(2\hk - 1)} a^{\frac{H}{2}(4 \hk - 1)}c^{\frac{H}{2}(4 \hk - 1)}}  \dd a \dd b \dd c \\
	  &\leq K\int_{[0,t]^3} \frac{1}{b^{H(2\hk - 1)} a^{\frac{H}{2}(4 \hk - 1)}c^{\frac{H}{2}(4 \hk - 1)}}  \dd a \dd b \dd c.
\end{align*}
The $a$ term and $c$ term exponents are obviously less than $1$ for $H < \bH(\hk)$. For the $b$ term exponent, we have $H(2\hk -1) < \bH(\hk) (2\hk-1) < 1$ for all $\hk \geq 1$. So this integral is finite.
Now for $H \geq 1/2$, notice that $\bH(\hk) \geq 1/2$ if and only if $\hk = 1$. So it suffices to look at the case of $H \in [1/2, 2/3)$ and $\hk = 1$. Following from \cref{keyboundcase1}, we have
\begin{align*}
	 &K \int_{[0,t]^3} \frac{|\mu|}{b^{3H} (a+c)^{3H}}  \dd a \dd b \dd c \\
	 &\leq K \int_{[0,t]^3}\frac{b }{b^{3H} (a+c)^{3H}}  \dd a \dd b \dd c \\
	 &= K \int_{[0,t]^3}\frac{1}{b^{3H-1} (a+c)^{3H}}  \dd a \dd b \dd c
\end{align*}
which is finite, where we have used that $\int_{[0,1]^2} \frac{1}{(x+y)^\xi} \dd x \dd y < \infty$ if and only if $\xi < 2$. This completes \textbf{Case 2}.
\item Finally, we look at the integration over $D^2_{3,t}$. Following the notation in \Cref{lemma:boundingregions} \cref{lemma:boundingregionsiii}, we have that $a = s - r $ and $c = s' - r'$. This gives that $\lambda = a^{2H}$ and $\rho = c^{2H}$. Now utilizing the lemma, we have 
\begin{align}
	\int_{ D_{3,t}^2} \frac{|\mu| \lambda^{\hk  - 1} \rho^{\hk - 1}}{(\lambda \rho - \mu^2)^{(4 \hk - 1)/2}} \dd r \dd s \dd r' \dd s' \nonumber
	&\leq  K \int_{[0,t]^3} \frac{|\mu| a^{2H(\hk - 1)} c^{2H(\hk - 1)}}{(a^{2H}c^{2H})^{(4 \hk - 1)/2}}  \dd a \dd b \dd c \nonumber \\
	&=K \int_{[0,t]^3} \frac{|\mu| a^{2H(\hk - 1)} c^{2H(\hk - 1)}}{(ac)^{H(4 \hk - 1)}}  \dd a \dd b \dd c \nonumber \\
	&=K \int_{[0,t]^3} \frac{|\mu|}{(ac)^{H(2 \hk + 1)}}  \dd a \dd b \dd c. \label{keyboundcase3}
\end{align}
Now we have that 
\begin{align*}
	2 \mu = (a+b+c)^{2H} + b^{2H} - (a+b)^{2H} - (b+c)^{2H}.
\end{align*}
By \citep{hu2001self}, this is equivalent to 
\begin{align}
	2 \mu = 2H(2H-1)ac \int_0^1 \int_0^1 (b + au + cv)^{2H-2} \dd u \dd v.\label{intrepcase3}
\end{align}
Using Young's inequality, we have $(b + (au + cv)) \geq Kb^{\alpha} (au + cv)^{\beta} \geq Kb^{\alpha}(aucv)^{\beta/2}$, where $\alpha + \beta = 1$, and $\alpha,\beta > 0$. Notice that the exponent $2H-2$ is negative for any $H$. Thus we have 
\begin{align*}
	|\mu| &\leq K ac b^{\alpha(2H-2)} \int_0^1 \int_0^1 ((au)^{\beta/2}(cv)^{\beta/2})^{2H-2} \dd u \dd v \\
		&\leq K b^{2\alpha(H-1)} a^{\beta(H-1) + 1} c^{\beta(H-1) + 1}.
\end{align*}
Leading on from \cref{keyboundcase3}, we have
\begin{align}
	&K \int_{[0,t]^3} \frac{|\mu|}{(ac)^{H(2 \hk + 1)}}  \dd a \dd b \dd c \nonumber \\
	&\leq K \int_{[0,t]^3} \frac{b^{2\alpha(H-1)} a^{\beta(H-1) + 1} c^{\beta(H-1) + 1}}{(ac)^{H(2 \hk + 1)}}  \dd a \dd b \dd c \nonumber\\ 
	& =  K \int_{[0,t]^3} \frac{1}{b^{2\alpha(1-H)}(ac)^{H(2 \hk + 1-\beta)+\beta-1}}  \dd a \dd b \dd c. \label{secondkeyboundcase3}
\end{align}
Choose $\alpha = \beta = 1/2$. Then this is equal to 
\begin{align*}
	K \int_{[0,t]^3} \frac{1}{b^{(1-H)}(ac)^{H(2 \hk + \frac{1}{2})-\frac{1}{2}}}  \dd a \dd b \dd c.
\end{align*}
Clearly the $b$ term exponent is always less than $1$ for any $H$. Studying the $ac$ term exponent, we have 
\begin{align*}
	H\left(2 \hk + \frac{1}{2}\right)-\frac{1}{2} &< H(\hk)\left (2 \hk + \frac{1}{2}\right )-\frac{1}{2}  \\&= \frac{4\hk + 1}{4\hk - 1} - \frac{1}{2}
\end{align*}
which is less than $1$ for $\hk > 1$. For $\hk =1$, the $ac$ term exponent is less than $1$ for $H \in [0, 3/5)$. So now we must deal with the case of $\hk = 1$ and $H \in [3/5, 2/3)$ separately. Instead, consider $\hk = 1$ and $H \in [1/2, 2/3)$, as this range of $H$ will be more telling as to why the subsequent methodology will not work for $H < 1/2$. Leading on from \cref{secondkeyboundcase3}, we must find $\alpha, \beta > 0$ with $\alpha + \beta = 1$ such that 
\begin{align*}
	K \int_{[0,t]^3} \frac{1}{b^{2\alpha(1-H)}(ac)^{H(3-\beta)+\beta-1}}  \dd a \dd b \dd c
\end{align*}
is finite. Focusing on the $ac$ term exponent, we must demand that $H(3 - \beta) + \beta - 1 < 1$ or equivalently, $\beta < (2- 3H)/(1-H)$. Define
\begin{align*}
	\beta(H) := \frac{2 - 3H}{1-H}.
\end{align*}
Notice that $\beta(H) \in (0, 1]$ for $H \in [1/2, 2/3)$. So choose $\beta < \beta(H)$, which will guarantee that the $ac$ term exponent is less than 1. Lastly, observing the $b$ term exponent, we have $2\alpha(1-H) <2(1-H)$, which is less than $1$ for any $H \in [1/2, 2/3)$. This completes \textbf{Case 3}.
\end{enumerate}
\end{proof}
\begin{lemma}
\label{lemma:finitenessL2even}
\pdfbookmark[2]{\Cref{lemma:finitenessL2even}}{lemma:finitenessL2even}
Consider \cref{keyintegraleven}, the key integral from the proof of \Cref{thm:convergenceL2evenrenorm}. Namely, 
\begin{align}
	\int_{D_t^2} \frac{\gamma^2}{(1-\gamma^2)^{(4k' + 1)/2} \lambda^{k'+1/2} \rho^{k'+1/2}} \dd r \dd s \dd r' \dd s' \label{keyintegralevenappen}
\end{align}
where we recall $\gamma = \mu/(\sqrt{\lambda \rho})$ and $k' \in \naturals$. Then this integral is finite if $H \in (0, \frac{2}{4k' + 1})$ where $k' \in \naturals$. 
\end{lemma}
\begin{proof}
Notice that the integrand can be rewritten as
\begin{align*}
	\frac{\gamma^{2}}{(1-\gamma^2)^{(4k' + 1)/2} \lambda^{k'+1/2} \rho^{k'+1/2}} = \frac{\mu^{2} \lambda^{k'-1}\rho^{k'-1}}{(\lambda \rho - \mu^2)^{(4k'+1)/2}}.
\end{align*}
Thus, finiteness of \cref{keyintegralevenappen} is guaranteed if and only if
\begin{align*}
	\int_{ D_{i,t}^2} \frac{\mu^{2} \lambda^{k'-1}\rho^{k'-1}}{(\lambda \rho - \mu^2)^{(4k'+1)/2}} \dd r \dd s \dd r' \dd s'
\end{align*}
is finite for each $i =1, 2,3$.  Define $\tH(k') := \frac{2}{4k' + 1}$, which is the proposed upper bound on $H$. The following steps are very similar to the steps in the proof of \Cref{lemma:finitenessL2odd}, and so we will skip some details. In what follows, we will make use of these facts:
\begin{itemize}
\item We will denote $K$ to be a strictly positive constant whose value may change line by line.
\item $\tH(k')$ is strictly decreasing and its maximum is $\tH(1) = 2/5$. This means that, unlike the odd case, we do not need to consider the case of $H \geq 1/2$.
\end{itemize}
We now proceed case by case.
\begin{enumerate}[label = \textbf{Case \arabic*}:]
\item We start by considering the integration over $D^2_{1,t}$. Following the notation in \Cref{lemma:boundingregions} \cref{lemma:boundingregionsi}, we have that $a+b = s - r $ and $b + c = s' - r'$. This gives that $\lambda = (a+b)^{2H}$ and $\rho = (b+c)^{2H}$. Now utilizing the lemma, we have 
\begin{align}
	\int_{ D_{1,t}^2} \frac{\mu^2 \lambda^{k'  - 1} \rho^{k' - 1}}{(\lambda \rho - \mu^2)^{(4 k' + 1)/2}} \dd r \dd s \dd r' \dd s' 
	&\leq  K \int_{[0,t]^3} \frac{\mu^2 (a+b)^{2H(k' - 1)} (b+c)^{2H(k' - 1)}}{((a+b)^{2H}c^{2H} + a^{2H}(b+c)^{2H})^{(4k' + 1)/2}}  \dd a \dd b \dd c \nonumber \\
	&\leq  K \int_{[0,t]^3} \frac{\mu^2 (a+b)^{2H(k'- 1)} (b+c)^{2H(k' - 1)}}{((a+b)c)^{\frac{H}{2} (4k' + 1)}(a(b+c))^{\frac{H}{2} (4k' + 1)}} \dd a \dd b \dd c \label{youngsevencase1} \\
		&=  K \int_{[0,t]^3} \frac{\mu^2 \dd a \dd b \dd c}{(a+b)^{\frac{5H}{2}}(b+c)^{\frac{5H}{2}} a^{\frac{H}{2} (4k'+1)}c^{\frac{H}{2}(4k'+ 1)}} \label{keyboundevencase1}
\end{align}
where we used Young's inequality, $x^{1/2}y^{1/2} \leq \frac{x}{2} + \frac{y}{2}$ in \cref{youngsevencase1}. Now
\begin{align*}
	2 \mu &= \big ((a+b+c)^{2H} +  b^{2H} - a^{2H} - c^{2H} \big )
\end{align*}
which gives the trivial bound
\begin{align*}
		\mu^2 &\leq K \big( a^{4H} + b^{4H} + c^{4H} \big).
\end{align*}
Thus we need to show that 
\begin{align*}
	K \int_{[0,t]^3} \frac{a^{4H} + b^{4H} + c^{4H}}{(a+b)^{\frac{5H}{2}}(b+c)^{\frac{5H}{2}} a^{\frac{H}{2} (4k' + 1)}c^{\frac{H}{2}(4k' + 1)}} \dd a \dd b \dd c
\end{align*}
is finite. So we look at each of the three terms in the numerator involving $a^{4H}, b^{4H}$ and $c^{4H}$ separately. First start with the $b^{4H}$ term. Then,
\begin{align*}
	&K \int_{[0,t]^3} \frac{ b^{4H}}{(a+b)^{\frac{5H}{2}} (b+c)^{\frac{5H}{2}} a^{\frac{H}{2} (4 k' + 1)}c^{\frac{H}{2} (4 k' + 1)}} \dd a \dd b \dd c \\
	&\leq K \int_{[0,t]^3} \frac{b^{4H}}{b^{\frac{5H}{2}} b^{\frac{5H}{2}} a^{\frac{H}{2} (4 k' + 1)}c^{\frac{H}{2} (4 k' + 1)}} \dd a \dd b \dd c \\
	&= K \int_{[0,t]^3} \frac{1}{b^{H} a^{\frac{H}{2} (4 k' + 1)}c^{\frac{H}{2} (4 k' + 1)}} \dd a \dd b \dd c
\end{align*}
which is clearly finite for $H< \tH(k')$. For the $a^{4H}$ term we have  
\begin{align*}
	&K \int_{[0,t]^3} \frac{ a^{4H}}{(a+b)^{\frac{5H}{2}} (b+c)^{\frac{5H}{2}} a^{\frac{H}{2} (4 k' + 1)}c^{\frac{H}{2} (4 k' + 1)}} \dd a \dd b \dd c \\
	&= K \int_{[0,t]^3} \frac{1}{(a+b)^{\frac{5H}{2}} (b+c)^{\frac{5H}{2}}a^{\frac{H}{2} (4 k' - 7)}c^{\frac{H}{2} (4 k' + 1)}} \dd a \dd b \dd c \\
	& \leq K \int_{[0,t]^3} \frac{1}{a^{\frac{5H}{2}} b^{\frac{5H}{2}}a^{\frac{H}{2} (4 k' - 7)}c^{\frac{H}{2} (4 k' + 1)}} \dd a \dd b \dd c \\
	&= K \int_{[0,t]^3} \frac{1}{b^{\frac{5H}{2}} a^{\frac{H}{2} (4 k' - 2)}c^{\frac{H}{2} (4 k' + 1)}} \dd a \dd b \dd c.
\end{align*}
Clearly the $c$ term exponent is less than 1 for $H < \tH(k')$. Looking at the $b$ term exponent, we have that $H < \tH(k') < \tH(1) = 2/5$, and so the $b$ term exponent is less than 1. Observing the $a$ term exponent, we have 
\begin{align*}
	\frac{H}{2} (4 k' - 2) < \frac{\tH(k')}{2} (4k' - 2) = \frac{4 k'- 2}{4k' + 1} < 1
\end{align*}
for any $k' \geq 1$. So this integral is finite. For the $c^{4H}$ term, notice that there is symmetry between $a$ and $c$ in the integrand of \cref{youngsevencase1}. Thus, this case is identical to the case of $a^{4H}$. Hence, we have completed \textbf{Case 1}. 
\item Next we look at the integration over $D^2_{2,t}$. Following the notation in \Cref{lemma:boundingregions} \cref{lemma:boundingregionsii}, we have that $a+b +c = s - r $ and $b = s' - r'$. This gives that $\lambda = (a+b+c)^{2H}$ and $\rho = b^{2H}$. Now utilizing the lemma, we have 
\begin{align}
	\int_{ D_{2,t}^2} \frac{\mu^2 \lambda^{k'  - 1} \rho^{k' - 1}}{(\lambda \rho - \mu^2)^{(4k' + 1)/2}} \dd r \dd s \dd r' \dd s' \nonumber
	&\leq  K \int_{[0,t]^3} \frac{\mu^2 (a+b+c)^{2H(k' - 1)} b^{2H(k' - 1)}}{(b^{2H} (a^{2H}+ c^{2H}))^{(4k' + 1)/2}}  \dd a \dd b \dd c \\
	&= K \int_{[0,t]^3} \frac{\mu^2 (a+b+c)^{2H(k' - 1)}}{b^{H(2k' + 3)} (a^{2H}+ c^{2H})^{(4k' + 1)/2}}  \dd a \dd b \dd c  \nonumber \\
	&\leq  K \int_{[0,t]^3} \frac{\mu^2 (a+b+c)^{2H(k' - 1)}}{b^{H(2k' + 3)} (a+c)^{H(4k' + 1)}}  \dd a \dd b \dd c \label{keyboundevencase2}
\end{align}
where we have used $K_1(x^p + y^p) \leq (x+y)^p \leq K_2(x^p + y^p)$ in \cref{keyboundevencase2}. 
Utilizing arguments from \textbf{Case 2} of \Cref{lemma:finitenessL2odd}, we have $\mu^2 \leq K b^{4H}$. Moreover, since $H < \tH(k') < \tH(1) = 2/5 < 1/2$, this bound on $\mu^2$ is always valid. Leading on from \cref{keyboundevencase2}, we have
\begin{align*}
	 K \int_{[0,t]^3} \frac{\mu^2 (a+b+c)^{2H(k' - 1)}}{b^{H(2k' + 3)} (a+c)^{H(4k' + 1)}}  \dd a \dd b \dd c
	 &\leq K\int_{[0,t]^3} \frac{b^{4H}(a+b+c)^{2H(k' - 1)}}{b^{H(2k' + 3)} (a+c)^{H(4k + 1)}}  \dd a \dd b \dd c \\ 
	  &\leq K\int_{[0,t]^3} \frac{(a+b+c)^{2H(k' - 1)}}{b^{H(2k' - 1)} a^{\frac{H}{2}(4k' + 1)}c^{\frac{H}{2}(4k' + 1)}}  \dd a \dd b \dd c \\
	  &\leq K\int_{[0,t]^3} \frac{1}{b^{H(2k' - 1)} a^{\frac{H}{2}(4k' + 1)}c^{\frac{H}{2}(4k' + 1)}}  \dd a \dd b \dd c.
\end{align*}
The $a$ term and $c$ term exponents are obviously less than $1$ for $H <\tH(k')$. For the $b$ term exponent, we have $H(2k' -1) < \tH(k') (2k'-1) < 1$ for all $k' \geq 1$. So this integral is finite. This completes \textbf{Case 2}.
\item Finally, we look at the integration over $D^2_{3,t}$. Following the notation in \Cref{lemma:boundingregions} \cref{lemma:boundingregionsiii}, we have that $a = s - r $ and $c = s' - r'$. This gives that $\lambda = a^{2H}$ and $\rho = c^{2H}$. Now utilizing the lemma, we have 
\begin{align}
	\int_{ D_{3,t}^2} \frac{\mu^2 \lambda^{k' - 1} \rho^{k' - 1}}{(\lambda \rho - \mu^2)^{(4 k' + 1)/2}} \dd r \dd s \dd r' \dd s' \nonumber
	&\leq  K \int_{[0,t]^3} \frac{\mu^2 a^{2H(k' - 1)} c^{2H(k' - 1)}}{(a^{2H}c^{2H})^{(4 k' + 1)/2}}  \dd a \dd b \dd c \nonumber \\
	&=K \int_{[0,t]^3} \frac{\mu^2 a^{2H(k' - 1)} c^{2H(k' - 1)}}{(ac)^{H(4k' + 1)}}  \dd a \dd b \dd c \nonumber \\
	&=K \int_{[0,t]^3} \frac{\mu^2}{(ac)^{H(2k'+ 3)}}  \dd a \dd b \dd c. \label{keyboundevencase3}
\end{align}
By \citep{hu2001self}, we have
\begin{align}
	2 \mu = 2H(2H-1)ac \int_0^1 \int_0^1 (b + au + cv)^{2H-2} \dd u \dd v.\label{intrepevencase3}
\end{align}
Using Young's inequality, we have $(b + (au + cv)) \geq Kb^{\alpha} (au + cv)^{\beta} \geq Kb^{\alpha}(aucv)^{\beta/2}$, where $\alpha + \beta = 1$, and $\alpha,\beta > 0$. Notice that the exponent $2H-2$ is negative for any $H$. Thus we have 
\begin{align*}
	|\mu| &\leq K ac b^{\alpha(2H-2)} \int_0^1 \int_0^1 ((au)^{\beta/2}(cv)^{\beta/2})^{2H-2} \dd u \dd v \\
		&\leq K b^{2\alpha(H-1)} a^{\beta(H-1) + 1} c^{\beta(H-1) + 1}
\end{align*}
which yields
\begin{align*}
		\mu^2 &\leq K  b^{4\alpha(H-1)} a^{2\beta(H-1) + 2} c^{2\beta(H-1) + 2}.
\end{align*}
Leading on from \cref{keyboundevencase3}, we have
\begin{align*}
	&K \int_{[0,t]^3} \frac{\mu^2}{(ac)^{H(2k'+ 3)}}  \dd a \dd b \dd c\\
	&\leq K \int_{[0,t]^3} \frac{b^{4\alpha(H-1)} a^{2\beta(H-1) + 2} c^{2\beta(H-1) + 2}}{(ac)^{H(2k' + 3)}}  \dd a \dd b \dd c \\ 
	& =  K \int_{[0,t]^3} \frac{1}{b^{4\alpha(1-H)}(ac)^{H(2 k' + 3-2\beta)+2\beta-2}}  \dd a \dd b \dd c.
\end{align*}
Choose $\alpha = 1/4$ and $\beta = 3/4$. Then this is equal to 
\begin{align*}
	K \int_{[0,t]^3} \frac{1}{b^{(1-H)}(ac)^{H(2k' + \frac{3}{2})-\frac{1}{2}}}  \dd a \dd b \dd c.
\end{align*}
Clearly the $b$ term exponent is always less than $1$ for any $H$. Studying the $ac$ term exponent, we have 
\begin{align*}
	H\left(2 k' + \frac{3}{2}\right)-\frac{1}{2} &< \tH(k')\left (2k' + \frac{3}{2}\right )-\frac{1}{2}  \\&= \frac{4k'+ 3}{4k' + 1} - \frac{1}{2}
\end{align*}
which is less than $1$ for $k' \geq 1$. This completes \textbf{Case 3}.
\end{enumerate}
\end{proof}

\end{document}